\documentclass[12pt,a4paper]{article}
\usepackage[left=1in,right=1in,top=1in,bottom=1in]{geometry}
\usepackage{graphicx}
\usepackage{amsmath}
\usepackage{amsthm,amssymb}
\usepackage{amsfonts}
\usepackage{verbatim}

\newcommand{\R}{\mathcal{R}} 
\newcommand{\N}{\mathcal{N}} 
\newcommand{\NR}{\bar{\mathcal{R}}} 
\newcommand{\G}{\mathcal{G}} 
\newcommand{\M}{\mathcal{M}} 
\newcommand{\I}{\mathcal{I}}
\newcommand{\NI}{\mathcal{NI}}
            
\newcommand{\nil}{\varnothing} 

\newtheorem{lemma}{Lemma}
\newtheorem{theorem}{Theorem}

\newtheorem{acknowledgement}{Acknowledgements}

   \newtheorem{cor}[theorem]{Corollary} 
     
\begin{document} 
%\titl{%n the structure of the set of generators $\pmod{p}$}
\title{On a cyclic structure of generators modulo primes} 
%orm $2^iq_1^{j_1}q_2^{j_2}+1$}
\author{Srikanth ch \thanks{Address: sricheru1214@gmail.com}, 
Shivarajkumar \thanks{Address: shivarajmacs@nie.ac.in} }  
\date{}
\date{Department of Mathematics,\\ National Institute of Engineering Mysore, 570008}
\vspace{2em}
\maketitle
\begin{abstract}
\noindent In this paper, we introduce a new notion called the
\textit{set of missing generators} $\mathcal{M}(g)$ for a generator 
(or primitive
element) $g$ of the cyclic group $\mathbb{Z}_p^*$, where
$p$ is an odd prime. The cardinality of
$\mathcal{M}(g)$ is established for all odd primes $p$. For primes $p$ of the form $2^iq_1^{j_1}q_2^{j_2}+1$, the collection $V_p = \{ \mathcal{M}(g):g\in \G \}$ forms an equinumerous partition of 
$\G$ (the set of all generators of  $\mathbb{Z}_p^*$), and a digraph defined on the vertex set 
$V_p$ is a disjoint collection of unicycles of the same size. 
Thus, for every such prime, an unique triplet $(c,n,e)$ of 
integers, describing the structure of the digraph of
missing generators, can be associated. With the
help of cyclic structure, we present a macroscopic additive 
property of generators of $\mathbb{Z}_p^*$. Further, we show that 
factoring RSA numbers is computationally equivalent to  
computing $T(p)$, under the assumption that there exists an
absolute constant $k$ such that  
the set $\{2^iN^j+1: 1\leq i,j<\log^k N\}$ contains a 
prime for any given odd $N$. 
\end{abstract}
\section{Introduction}  
\noindent For a prime $p$, $\mathbb{Z}_p^*$ is a cyclic group under multiplication. 
There exist $\phi(p-1)$ generators (or primitive elements) in the group, 
where $\phi$ is the Euler's totient function.  \\

\noindent \textbf{Definition :} An integer $a$ is called a quadratic residue$\pmod{p}$ 
if the congruence 
$x^2 \equiv a \pmod p$ has solutions. Otherwise, $a$ is 
called a quadratic non residue(simply non residue).  \\

In $\mathbb{Z}_p^*$, let $\mathcal{R}$ be the set of all quadratic residues. Let 
$\mathcal{N}$ 
be the set of all non residues. 
Let $\mathcal{G}$ be the set of generators. 
Let $\mathcal{NG}$ represent the set of all non residues,  
which are not generators. 
The sets have the following well-known relations. (For standard notions of the structure  
of $\mathbb{Z}_p^*$, we refer to \cite{apostol}).   
\begin{itemize}
 \item $\mathcal{R} \cap \mathcal{N} = \emptyset$,
 $\mathbb{Z}_p^{*}= \mathcal{R} \cup \mathcal{N}$.
 \item 
 $\mathcal{G} \cap \mathcal{NG} = \emptyset$, $\mathcal{N} = \mathcal{G} \cup \mathcal{NG}$.
 \item $|\mathcal{R}| = |\mathcal{N}| = \frac{p-1}{2}$.
 \item $|\mathcal{G}|=\phi(p-1)$. 
\end{itemize}
 
Further, the product of a residue and a non residue results in a non residue$\pmod{p}$. 
Thus, for any $g\in \mathcal{G}$, the set $\{gr\pmod{p}: r\in \R\}$ is same as $\mathcal{N}$. \\ 

The starting point of the present study is one simple observation that, with respect to a generator $g$, the set $\R$ can be viewed as the union of 
two disjoint subsets $\R_g$ and $\bar{\mathcal{R}}_g$, 
where $\R_g = \{r \in R: gr \in \G \}$ and  $\bar{\mathcal{R}}_g = \{r \in \R : gr \in \mathcal{NG}\}$. 
 Further, we define the following sets.  
\begin{eqnarray}
 \mathcal{I}(g) & = & \R_g \cap \R_{g^{-1}}  \nonumber \\
 \mathcal{NI}(g) & = & \bar{\mathcal{R}}_g \cap \bar{\mathcal{R}}_{g^{-1}}  \nonumber 
\end{eqnarray}

By definition, for any $r\in \mathcal{I}(g)$, 
$gr,g^{-1}r \in \mathcal{G}$. So, every element of the  set $\{gr,g^{-1}r| r\in \mathcal{I}(g)\}$ is a 
generator. But, 
this set may not contain all the elements of $\G$. The generators that are missing in this 
set are called \textit{missing generators} for $g\in \mathcal{G}$. The formal definition is as follows.   
\begin{displaymath}
\mathcal{M}(g) =  \mathcal{G}\setminus \{gr,g^{-1}r\lvert  r\in \mathcal{I}(g)\}. 
\end{displaymath}
 In words, the set $\mathcal{M}(g)$ consists of generators that can not be expressed as the product of 
 any residue from $\mathcal{I}(g)$ and either $g$ or $g^{-1}$. 
 With this observation, we are motivated to  
 study the properties of $\mathcal{M}(g)$. 
 In the study of these properties, we first prove that the cardinality of $\mathcal{M}(g)$ is same for any 
 $g\in \mathcal{G}$. 
 
 On the other hand, we have defined the set $\mathcal{NI}(g)$, which consists of 
 residues $r$ such that $gr$, $g^{-1}r\in \mathcal{NG}$. 
 Like $\mathcal{M}(g)$, the cardinality of $\mathcal{NI}(g)$ 
 is same for any $g\in \mathcal{G}$. However, the two sets are  
 entirely different: one set consists of generators and 
 the other set consists of residues.
  
 \par Define $P_n$ to be the set all primes such that, \ \ \ \ .
 We prove a connection between the two sets for primes $p \in P_3 = \{2^iq_1^{j_1}q_2^{j_2}+1: i,j_1,j_2\geq 1\}$. For the same 
 primes, we present a macroscopic property of 
 additive inverses of generators with help of the
 cyclic structure of missing generators explained in the following section. 
 
 \subsection{Results}
  For prime $p$, let $M_p$ 
 denote the cardinality of $\mathcal{M}(g)$ in $\mathbb{Z}_p^*$, and let 
 $N_p$ denote the cardinality of $\mathcal{NI}(g)$. 
 We prove the following. 

\begin{theorem}
 For prime $p$, with $p-1 = 2^s\prod_{j=1}^k q_j^{\alpha_j}$, $q_j$ odd prime
 \begin{eqnarray}
  M_p & = & \frac{p-1}{2\prod_{j=1}^k q_j}\bigg[ \prod_{j=1}^k (q_j-1) - 2\prod_{j=1}^k (q_j-2) + \prod_{j=1}^k (q_j-3) \bigg], \nonumber \\
  N_p & = & \frac{p-1}{2\prod_{j=1}^k q_j}\bigg[ \prod_{j=1}^k q_j - 2\prod_{j=1}^k (q_j-1) + \prod_{j=1}^k (q_j-2) \bigg]. \nonumber
  \end{eqnarray}
\end{theorem}
As an immediate consequence to the above theorem we have 
the following corollary comparing the cardinality of the two sets, subject to the number of prime factors of $p-1$.  
\begin{cor}
  For odd prime $p$, the following holds.
\begin{itemize}
 \item $N_p = M_p = 0$, if $p-1$ has at most two prime factors 
 \item $N_p = M_p > 0$, if $p-1$ has exactly three prime factors
 \item $N_p > M_p > 0$, if $p-1$ has at least four prime factors. 
 \end{itemize}
\end{cor}

From the above corollary, for prime $p$ of the form $2^sq_1^{j_1}q_2^{j_2}+1$, $M_p = N_p = 2^sq_i^{j_1-1}q_2^{j_2-1}$.
As noted earlier, this equality motivates us 
to further study sets of missing generators and sets $\mathcal{NI}$ for prime $p$ with $p-1$ has exactly three prime 
factors. 
\subsubsection{Cyclic structure of missing generators}
For primes $p \in P_3 = \{2^iq_1^{j_1}q_2^{j_2}+1: i,j_1,j_2\geq 1\}$, 
we prove the following properties of missing generators. 
\begin{itemize}
 \item The set $\{\mathcal{M}(g): g\in\mathcal{G}\}$ forms an equinumerous partition of $\G$.
 %\item Suppose $\G_1, \G_2,\ldots, \G_t$ is the equinumerous partition of $\G$, resulting from $\{\mathcal{M}(g): g\in\mathcal{G}\}$. Then, for any $g\in \G_j$, $1\leq j\leq t$, $\mathcal{M}(g)$
 %is same.  
 \item For any $A\in \{\mathcal{M}(g): g\in\mathcal{G}\}$, the following holds: 
 $\forall \ x\in A$, $\mathcal{M}(x)$
 is same.
 
 \end{itemize}
A directed graph (simply digraph) $D$ is a finite nonempty set of objects called vertices together with a set of ordered pairs of distinct vertices of $D$ called edges. 

\par With the above properties, one can define a  digraph $D=(V,E)$ where 
 $ V = \{\mathcal{M}(g):g\in\G \}$ is the vertex set, 
 and $E$ is the edge set. There is an edge from $\mathcal{M}(g_i)$ to $\mathcal{M}(g_j)$ if and only if 
 $g_j\in \mathcal{M}(g_i)$. Both in-degree and out-degree 
 of each vertex is $1$. Thus, the digraph $D$ is a 
 collection of disjoint unicycles and each unicycle has the same number of vertices. 
 In the following examples, we illustrate the unicycles 
 of missing generators.  
\begin{itemize}
\item  For prime $p=31$, $\G = \{3,21,24,22,13,12,17,11\}$. 
Missing generator sets are:   $\mathcal{M}(17) = \mathcal{M}(11) = \{13,12\}$, 
$\mathcal{M}(13) = \mathcal{M}(12) = \{24,22\}$
, $\mathcal{M}(24) = \mathcal{M}(22) = \{3,21\}$ and $\mathcal{M}(3) =\mathcal{M}(21) = \{17,11\}$. The 
digraph of missing generators has only one unicycle consisting of 4 vertices. There are two generators in each vertex.  
 
\begin{figure}[h]
\centering
\includegraphics[width=0.6in]{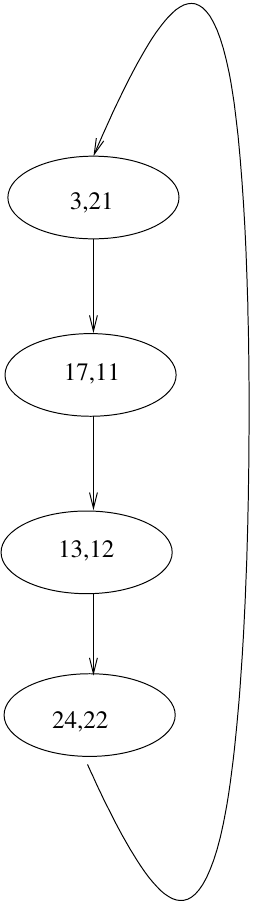}
\caption{unicycle of Missing generators for prime 31}
\end{figure}

\item For $p =43$, $\G = \{3,5,12,18,19,20,26,28,29,30,33,34\}$. Missing generator sets are: 
$\mathcal{M}(28)=\mathcal{M}(20)=\{5,26\}$, $\mathcal{M}(5) = \mathcal{M}(26) =\{3,29\}$, 
$\mathcal{M}(3) = \mathcal{M}(29) = \{28,20\}$, $\mathcal{M}(12) = \mathcal{M}(18) = \{19,34\}$, 
$\mathcal{M}(19) = \mathcal{M}(34) = \{30,33\}$ and
$\mathcal{M}(30) = \mathcal{M}(33) =\{12,18\}$. The digraph of missing generators has two unicycles and each unicycle has three vertices. 
 
\begin{figure}[h]
\centering
\includegraphics[width=1.4in]{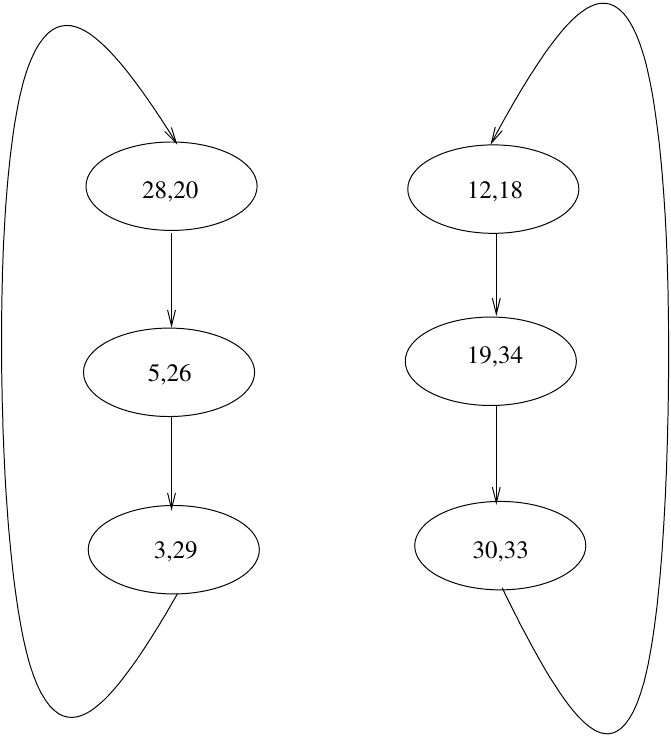}
\caption{unicycles of missing generators for prime 43}
\end{figure}

\end{itemize}

With the above observation, one can define a map T that uniquely associates a prime $p\in P_3$ with a triplet $(c,n,e)$ of numbers that describes the 
cyclic structure of missing generators in 
$\mathbb{Z}_p^*$. Here, $c$ denotes the number of 
unicycles, $n$ is the number of vertices in each unicycle, and 
$e$ is the number of generators in each
vertex. For 
prime $p=2^sq_1^{j_1}q_2^{j_2}+1$, the numbers 
follow the relation: $c\times n \times e=\phi(p-1)$, $e = 2^iq_1^{j_1-1}q_2^{j_2-1}$ and $4^n\equiv 1\pmod{q_1q_2}$. 
For example, for $p=31$, $(c,n,e) = (1,4,2)$ (Figure 1),
and for $p=43$, $(c,n,e) = (2,3,2)$ (Figure 2).

\subsubsection{\bf Connection between sets $\mathcal{M}(g)$ and sets $\mathcal{NI}$} 
\noindent  Like sets of missing generators, sets $\NI$
have the following properties.    
\begin{enumerate}
\item For any $g_1,g_2 \in \G$, exactly one of the 
following holds. 
\begin{itemize}
 \item $\mathcal{NI}(g_1) \cap \mathcal{NI}(g_2) = \emptyset$
  \item $\mathcal{NI}(g_1) = \mathcal{NI}(g_2)$.
\end{itemize}
\item For any $g_1,g_2\in \mathcal{M}(g)$, $\mathcal{NI}(g_1) = \mathcal{NI}(g_2)$.
\end{enumerate}
Due to the second property, we can call
$\mathcal{NI}(g_i)$, for 
$g_i\in \mathcal{M}(g)$, as the set 
$\mathcal{NI}$ \textit{corresponding to} $\mathcal{M}(g)$.

\subsubsection*{Additive Inverses of generators}
For primes of the form $4k+1$, for any $g\in \G$, 
both $g$ and $-g$ belongs to the same set of missing generators. 
For primes of the form $4k+3$, additive inverses of 
generators belongs to sets $\mathcal{NI}$. The additive property is elaborated as follows.  

For prime $p = 2^sq_1^{j_1}q_2^{j_2}+1$ of the 
form $4k+3$, suppose that $(c,n,e)$ 
is the triplet describing the digraph ($D$) of missing 
generators. 
Let $L = \{1,2,\ldots, c\}$ be a labeling (numbering)
of the unicycles of $D$. 
A relation S on $L$ is defined as follows: If the additive inverses of all generators in the 
 $i^{th}$ unicycle belongs to the sets $\mathcal{NI}$
 corresponding to the vertices of the 
 $j^{th}$ unicycle then $(i,j) \in S$. The relation $S$ is either reflexive or symmetric, subject to 
the conditions:  
\begin{itemize}
 \item  $S$ is reflexive if and only if 
 $2^n \equiv 1 \pmod {q_1q_2}$, where$k_1$ and $k_2$.
 $n$ is the number of vertices in each unicycle of 
 $D$. 
 \item  $S$ is symmetric if and only if either $n$ is odd or $2^n \equiv -1 \pmod{q_1q_2}$ for even $n$.
\end{itemize}
The relation S gives the additive property of all generators 
belonging to a cycle. In this sense, we call this a 
macroscopic property. 

\subsubsection*{Connection to integer factoring}

The Integer Factoring Problem (IFP) is the problem of 
finding the prime factors of a given integer $x$. The state-of-art classical technique for solving the problem is Number Field Sieve (NFS) algorithm whose running time complexity is sub-exponential\footnote{The running time complex of NFS is 
$exp((\log x)^{\alpha} (\log\log x) ^{1-\alpha})$ for $\alpha=1/3$.} in 
$\log x$ if $x$ is an worst-case instance \footnote{An worst-case instance of IFP is a large number N which is the product two 
equally sized primes $p$ and $q$, that is $p,q \approx \sqrt{N}$.} of the problem which form the basis for the security of RSA cryptographic system. However, Shor's quantum algorithm 
\cite{shor} solves the problem can be solved in $\log^3 x$ time. 

We show a computational equivalence 
between the problem of computing triplet (for a given $p$) and 
the problem of factoring a given RSA number, under the assumption: there exist
absolute constant $k$ such that  
the set $\{2^ir^j+1: 1\leq i,j<\log^k r\}$ contains a 
prime for any given odd $r$. The computational gap between 
the two problems is proved to be $O(\log ^{4k+3} N)$ for 
a given RSA number $N$. \\

 \noindent {\bf Organization of the paper:} In Section \ref{sec_proof}, the cardinality of sets 
 $\mathcal{M}(g)$ and $\mathcal{NI}(g)$ are established. In Section \ref{sec_cyclic}, the cyclic structure of 
 missing generators is studied. In Section \ref{sec_ni}, 
  the additive property of generators is presented with the help of cyclic structure of missing generators and 
  the relationship between $\mathcal{NI}$ and sets 
 $\mathcal{M}(g)$. In Section \ref{sec_t}, the map T is studied.
 and equivalence between the problem of computing triplets and IFP is studied.  The reader is requested to be informed that prime $p\in P_3$ starting Section \ref{sec_cyclic}
 for the rest of the paper .
 
\section{Proof of Theorem 1.1}\label{sec_proof} 
For an integer $n>0$, define 
\begin{eqnarray}
\mathit{A}(n) & = & \{ x: 0<x<n, gcd(x,n)=1, gcd(x-2,n)>1, gcd(x+2,n)>1\},  \nonumber \\
\mathit{B}(n) & = & \{ x : 0<x<n, x\, even,  gcd(x-1,n)>1, gcd(x+1,n)>1 \}. \nonumber 
\end{eqnarray} 
The set $\mathcal{M}(g)$ has the following property. 
\begin{lemma}
 $g^x\in \mathcal{M}(g)$ if and only if $x\in A(p-1)$. 
\end{lemma}
\begin{proof}
For $x \in A(p-1)$,  $g^{x+2},g^{x-2} \not\in \G$. 
Thus $g^{x+1},g^{x-1} \not \in \mathcal{I}(g)$.
 Therefore, $g^x \in \mathcal{M}(g)$. 
Conversely, suppose that, $g^x \in \mathcal{M}(g)$, 
then $g^{x+1}, g^{x-1} \not \in \mathcal{I}(g)$. 
It is true only when $g^{x+1} \in \R_{g^{-1}}$, $g^{x-1} \in \R_g$, 
$g^{x-1} \in \bar{\R}_{g^{-1}}$ and $g^{x+1} \in \bar{\R}_g$. So, $g^{x+2}, g^{x-2} \in \N\G$.  
Therefore, $x\in A(p-1)$.   
\end{proof} 
\noindent It is easy to see that, the set $\mathcal{NI}(g) $ has the following property. 
 \begin{lemma}
  $g^{x} \in \mathcal{NI}(g)$ if and only if $x\in B(p-1)$.
 \end{lemma} 
 \begin{proof}
The proof follows from the definition of $\mathcal{NI}(g)$.   
\end{proof} 
 
As a corollary to Lemma 1, we have $|\mathcal{M}(g)| = |A(p-1)|$, for any $g\in \G$. Similarly, 
from Lemma 2, $|\mathcal{NI}(g)|= |B(p-1)|$
for any $g\in \G$. Now, we establish the cardinality of $A(n)$, for integer $n>1$. 
This general result proves the value of $M_p$ stated in Theorem 1.1. Proving $|B(n)|$ 
takes similar arguments as the proof of $|A(n)|$. We thus skip the proof of $|B(n)|$. In establishing $|A(n)|$, we 
first prove the following.
\begin{lemma} 
For $n = 2^s\prod_{i=1}^k q_i^{\alpha_i}$, $q_i$ odd prime, $|A(n)| = \frac{n}{z} |A(z)|$, where $z=2\prod_{i=1}^k q_i$.
\end{lemma}
\begin{proof}
For any $x\in A(z)$ and 
for all $r = 0,1,2,\ldots \frac{n}{z}-1$, $x+rz \in A(n)$,.  
\end{proof} 

\begin{comment}
 From the definition, for any $x\in A(z)$, $x^2-4$ 
is a product of at least two distinct odd prime factors of $z$. We remember this 
property in proving $|A(z)|$.
\end{comment}

From Lemma 3, it is enough to establish the cardinality of $A(z)$ for 
 $z=2\prod_{i=1}^k q_i$. To prove this, we introduce 
 the following notation. 
\begin{itemize}
 \item For  $2\leq j\leq k$, let $A(z_j)$ represent the set of 
all $x\in A(z)$ such that $x^2-4$ is a product of $j$ odd prime factors of $z$. 
\item $\omega(y)$ represent the number of distinct odd prime factors of $y$. 
\item $N(z,j)$ represents the sum of all divisors $y$ of $z$, with $\omega(y)=j$. 
Let $N(z,0)=1$. 
\end{itemize}

\begin{lemma}
 For $z=2\prod_{i=1}^k q_i$, $|A(z)| = \sum_{j=0}^{k-2} (-1)^{k-2-j} a_jN(z,j)$, where $a_j = 3^{k-j}-2^{k-j+1}+1$. 
\end{lemma}
\begin{proof}
Suppose that, $y$ is a product of $j$ 
odd prime factors of $z$. Let $a$ be 
a non-trivial square-root of $4\pmod{y}$. Then, each $a+ry$, which is 
relatively prime to $z$, and less than $z$, belongs to $A(z_j)$. Using the inclusion-exclusion principle, 
the number of such $a+ry$ can be deduced to be  
$N(z/y,k-j) - N(z/y,k-j-1) + N(z/y,k-j-2) - \ldots + (-1)^{k-j} N(z/y,0)$. 
Since there are $2^j-2$ non-trivial square roots of $4\pmod{y}$, the total 
number of elements in $A(z_j)$ corresponding to the product $y$ is 
\begin{displaymath}
(2^j-2)\big(N(z/y,k-j) - N(z/y,k-j-1) + N(z/y,k-j-2) - \ldots + (-1)^{k-j} N(z/y,0)\big).
\end{displaymath}
Thus,  
\begin{eqnarray}
|A(z_j)|  & = & (2^j-2) \sum_{\stackrel{y|z,}{\omega(y) =j.}}  \bigg(N(z/y,k-j) - 
                     N(z/y,k-j-1) + \ldots +(-1)^{k-j} N(z/y,0)\bigg)  \nonumber \\
            & = & (2^j-2) \sum_{b=j}^k (-1)^{b-j} \binom{b}{j}N(z,k-b). \nonumber 
\end{eqnarray}
\noindent Using the inclusion-exclusion principle, we have 
\begin{eqnarray}
 |A(z)|& = & \sum_{j=2}^k (-1)^j |A(z_j)| \nonumber \\
       & = & \sum_{j=0}^{k-2} (-1)^{k-2-j} a_jN(z,j)  \,\, (\textrm{after simplification}) 
\end{eqnarray}
where $a_j = \sum_{r=2}^{k-j} (2^{r}-2)\binom{k-j}{r} =  3^{k-j}-2^{k-j+1}+1$.  
\end{proof}
One can that, the sum 
\begin{displaymath}
 \prod_{i=1}^k (q_i-1) - 2 \prod_{i=1}^k (q_i-2) + \prod_{i=1}^k (q_i-3)\, \textrm{simplifies to the expression for $|A(z)|$}.
\end{displaymath}
 This completes the proof of $M_p$.  
\section{Cyclic structure of missing generators}\label{sec_cyclic}
In this section, we prove that, sets of missing generators exhibit a cyclic behavior for 
prime $p = 2^iq_1^{j_1}q_2^{j_2}+1$. We define below some sets, whose properties are essential ingredients for the proof. \\
 
Let $x_1, x_2$ be the non-trivial roots of the congruence 
$x^2 \equiv 4 \pmod {q_1q_2}$. Let $x_1$ be odd, and $x_2$ be even. 
We define  
\begin{displaymath}
 B_{x_1} = \{ x_1+q_1q_2s | s = 0,2,\ldots,2^iq_1^{j_1-1}q_2^{j_2-1}-2 \}
\end{displaymath}
 \begin{displaymath}
 B_{x_2} = \{ (x_2+q_1q_2r | r = 1,3,\ldots,2^iq_1^{j_1-1}q_2^{j_2-1}-1 \}.
\end{displaymath}
For $j$ relatively prime to $p-1$, define $B_{x_1}(j) = \{xj\pmod{p-1} : x \in B_{x_1}\}$ 
and  $B_{x_2}(j) = \{xj\pmod{p-1} : x \in B_{x_2}\}$.  It can be verified that the 
sets have the following properties: 
\begin{itemize}
 \item  $B_{x_1}\cap B_{x_2} = \emptyset$, $B_{x_1}\cup B_{x_2} = A(p-1)$
 \item  $x \in B_{x_1}$,  $p-1-x \in B_{x_2}$ since $x_1+x_2 =q_1q_2$
 \item  $B_{x}(j) = B_{xj}$, for $x\in \{x_1,x_2\}$
 \item  $B_a(j) = B_x(j)$, for any $a\in B_x(j)$.
\end{itemize}
\noindent From Lemma 1, we have: 
\begin{equation}
 \label{rel_bet_missing_B}
 g^x \in \mathcal{M}(g) \Leftrightarrow x \in B_{x_1} \cup B_{x_2}.
\end{equation}
We now prove the following properties of missing generators. 
\begin{lemma}
For any $g_1,g_2 \in \G$, exactly one of the following
is valid
\begin{itemize}
\label{partition}
 \item $\mathcal{M}(g_1)\cap \mathcal{M}(g_2) = \emptyset$
 \item $\mathcal{M}(g_1)= \mathcal{M}(g_2)$.
\end{itemize}
\end{lemma}
\begin{proof}
Let $g_2 \equiv g_1^k\pmod{p}$ for 
some $k$ relatively prime to $p-1$. From (2), we have 
$\mathcal{M}(g_1) = \{g_1^x: x \in B_{x_1} \cup B_{x_2}\}$ and 
$\mathcal{M}(g_2) = \{g_1^y:  y \in B_{x_1}(k) \cup B_{x_2}(k)\}$.  Suppose $g \in  \mathcal{M}(g_1) \cap \mathcal{M}(g_2)$. Then, 
$g \equiv g_1^x \equiv g_1^y\pmod{p}$ for some  $x \in B_{x_1} \cup B_{x_2}$ and 
$y\in B_{x_1}(k) \cup B_{x_2}(k)$. This implies that 
$x\equiv y\pmod{p-1}$. Equivalently, $x\equiv kx^{'}\pmod{p-1}$ for some $x^{'} \in B_{x_1} \cup B_{x_2}$. 
This implies that $k = \pm 1+q_1q_2r\pmod{p-1}$ for some $r$. Thus, 
$B_{x_1}(k)=B_{x_1}$, $B_{x_2}(k)=B_{x_2}$.  
\end{proof}

Thus, the sets of missing generators $\mathcal{M}(g)$ forms a partition of $\G$. The number of 
distinct elements in the partition is $\phi(p-1)/M_p$. 
Further, these sets have the following property. 

\begin{lemma}
\label{same_miss}
 For any $g_1,g_2 \in \mathcal{M}(g)$, $\mathcal{M}(g_1)=\mathcal{M}(g_2)$.
\end{lemma}
\begin{proof}
Suppose $g_1,g_2 \in \mathcal{M}(g)$. 
Let $g_1 = g^{x}\pmod{p}$ and $g_2 = g^{x^{'}}\pmod{p}$ for some  $x, x^{'} \in B_{x_1} \cup B_{x_2}$. From the 
properties of $B$ sets, $B_{x_1}(x) \cup B_{x_2}(x) =   B_{x_1}(x^{'}) \cup B_{x_2}(x^{'})$.  
\end{proof}

Thus, we can define a digraph $D=(V,E)$,where,
 $ V = \{\mathcal{M}(g):g\in\G \}$ and the edges defined as follows: 
 there is a directed edge from $\mathcal{M}(g_i)$ to $\mathcal{M}(g_j)$ if and only if 
 $g_j\in \mathcal{M}(g_i)$.  From Lemma \ref{partition} \& \ref{same_miss}, it can 
 be inferred that both in-degree and out-degree of each 
 vertex is $1$. 
 Thus, the digraph $D$ is a collection of disjoint unicycles. 
 \begin{theorem}
 Each unicycle of the digraph has same number of vertices.
\end{theorem}
\begin{proof}
Let $C$ be a unicycle of the digraph $D$. Suppose $g$ is 
a generator in a vertex of $C$. 
Then, from (\ref{rel_bet_missing_B}),  $C$ contains vertices $\mathcal{M}(g^{x_1})$, $\mathcal{M}(g^{x_1^2})$, 
$\mathcal{M}(g^{x_1^3})$, $\ldots$,. Suppose $n$ is 
the least positive integer such that $x_1^n = \pm 1+q_1q_2z \pmod{p-1}$.   
Then, $\mathcal{M}(g^{x_1^n}) = \mathcal{M}(g)$. Thus, $n$ is the the number of vertices of $C$. This property 
holds for any unicycle of $D$. 
\end{proof} 

\noindent The structure of the digraph is clear. It is a disjoint collection of unicycles of same size,  
and every vertex has the same number of generators. Thus,  we can associate a triplet $(c,n,e)$ 
with prime $p$ to describe the graph $D$ formed with $\mathcal{M}(g)$ of $\mathbb{Z}_p^*$. 
Here, $c$ represents the number of unicycles in $D$,  $n$ the number of vertices in a unicycle, 
and $e$ the number of generators corresponding to a vertex. These integers are such that  
$c\times n\times e = \phi(p-1)$ and  $e=|\mathcal{M}(g)| = 2^iq_1^{j_1-1}q_2^{j_2-1}$. 
In the following lemma, we can see that, the integer $n$ has the following property. 
\begin{lemma}
 If the digraph has unicycles of size $n$, then 
\begin{itemize}
 \item $2^n = \pm 1 \pmod {q_1q_2}$, if $n$ is even
 \item $2^{n-1} = \pm x_1^{-1} \pmod {q_1q_2}$, if $n$ is odd
\end{itemize}
\end{lemma}
\begin{proof}
For unicycles of size $n$, we have the congruence $x_1^n 
\equiv \pm 1+q_1q_2z \pmod {p-1}$, which implies $x_1^n \equiv \pm 1 \pmod {q_1q_2}$. 
If $n$ is even, $x_1^n \equiv  (x_1^2)^{\frac{n}{2}} \equiv 4^{\frac{n}{2}} \equiv \pm 1 \pmod {q_1q_2}$ 
and $2^n \equiv \pm 1 \pmod {q_1q_2}$. 
If $n$ is odd, $x_1^{n-1}.x_1 \equiv 2^{n-1}x_1 \equiv \pm 1 \pmod {q_1q_2}$.
 Therefore $2^{n-1} \equiv \pm x_1^{-1} \pmod {q_1q_2}$. 
\end{proof}

\section{Properties of $\mathcal{NI}$}\label{sec_ni}
\begin{comment}
Till now we have studied sets $\mathcal{M}(g)$. It has been shown that $\{M(g,g^{-1}:g\in G\}$ is a partition of $G$. Like sets of missing generators,
set $NI$ also follow the property: $\mathcal{NI}g_i,g_i^{-1}) \cap \mathcal{NI}g_j,g_j^{-1}) = \emptyset$ or $\mathcal{NI}g_i,g_i^{-1}) = \mathcal{NI}g_j,g_j^{-1})$ for any $g_i,g_j\in G$. 
This property will be shown here. So the set $\{\mathcal{NI}g,g^{-1}:g\in G\}$ is a partition of $\{r\in \mathcal{NI}g,g^{-1}:g\in G\}$ which is 
a subset of $R$ the set of residues. We see that that this subset contains all generators of $R$. Finally, with help of generalized
notion on ``set of missing generators'', it will be shown that sets $NI$ also exhibit the graph structure. \\

\end{comment} 

It has been shown that $\{\mathcal{M}(g):g\in \mathcal{G}\}$ is a partition of $\mathcal{G}$ and sets of missing
generators exhibit cycle structure. We shall see that sets $\NI$ also follow the same property. 
This is achieved through series of results in present section and the section following.
We start with that following. (Recall that $p\in P_3$ with $p-1 =2^iq_1^{j_1}q_2^{j_2}$). 

\begin{lemma}\label{lem_ni}
 For any $g_1,g_2\in \mathcal{M}(g)$, 
 $\mathcal{NI}(g_1) = \mathcal{NI}(g_2)$.   
\end{lemma}
\noindent Proof: 
Suppose that $g_1 =g^{s_1}$ and $g_2 = g^{s_2}$ belong to the same set 
$\mathcal{M}(g)$. 
So that, $s_2 \equiv s_1+zq_1q_2 \pmod {p-1}$ for some $z \in 
\{0,2,4,\ldots,2^{i-1}q_1^{j_1}q_2^{j_2}-2\}$. 
If $g^k \in \mathcal{NI}(g_1)$, then $gcd(k\pm s_1, p-1) >1$. 
This implies that, $gcd(k \pm s_2, p-1) >1$. Therefore, $r=g^k \in \mathcal{NI}(g_2)$. 
Similarly, if $r \in \mathcal{NI}(g_2)$, then $r \in \mathcal{NI}(g_1)$. 
Hence $\mathcal{NI}(g_1) = \mathcal{NI}(g_2)$.   \hfill $\Box$ \\

Let $D$ be the directed graph generated of missing generators.  Let  
$V = \{\mathcal{M}(g)| g\in \G\}$ be the vertex set of $D$. Let $V_1 = \{\mathcal{NI}(g)| g\in \mathcal{G}\}$. 
Let $\mathcal{NI}(\mathcal{M}(g))$ to be $\mathcal{NI}(g_i)$, $g_i\in \mathcal{M}(g)$. 
Define a map $f:V \rightarrow V_1$ as $f(\mathcal{M}(g)) = \mathcal{NI}(g_i)$, where $g_i\in \mathcal{M}(g)$.
 By Lemma \ref{lem_ni}, the map $f$ is well defined. 

\begin{lemma} 
 The map $f$ is bijective
\end{lemma}
\noindent \textbf{Proof:}  Suppose that, for $g^{k_1},g^{k_2},g^j\in \G$, 
$$f(\mathcal{M}(g^{k_1})) = f(\M(g^{k_2})) = \mathcal{NI}(g^j).$$
This implies that,
 $j\equiv k_1x_1+q_1q_2z_1 \equiv k_2x_1+q_1q_2z_2\pmod{p-1}$ for some even $z_1,z_2$
and $g^{k_2x_1} \in \M(g^{k_1})\cap \M(g^{k_2})$. By Lemma \ref{lem_ni}, 
 $\M(g^{k_1}) =\M(g^{k_2})$. So, $f$ is one-one.
 For $\mathcal{NI}(g)\in V_1$, $f(M(g^y))=\mathcal{NI}(g)$, where $y\equiv x_1^{-1}\pmod{p-1}$. So, $f$ is onto. 
\hfill $\Box$ \\
 
We now prove the following property of $\mathcal{NI}$. 

\begin{theorem}
 For any $g_1,g_2 \in \mathcal{G}$, exactly one of the following holds
\begin{itemize}
 \item $\mathcal{NI}(g_1) \cap \mathcal{NI}(g_2) = \emptyset$
 \item $\mathcal{NI}(g_1) = \mathcal{NI}(g_2)$
\end{itemize}
\end{theorem}

\noindent Proof: Suppose $g^j \in \mathcal{NI}(g_1) \cap \mathcal{NI}(g_2)$. Let $f(\M(g^{k_1})) = \mathcal{NI}(g_1)$
and $f(\M(g^{k_2})) = \mathcal{NI}(g_2)$. 
Then $j\equiv k_1x_1+q_1q_2z_1 \equiv k_2x_1+q_1q_2z_2
\pmod{p-1}$. This implies that,  
$g^{k_2x_1} \in \M(g^{k_1})\cap \M(g^{k_2})$ and therefore $M(g^{k_1}) = M(g^{k_2})$. Hence 
$f(\M(g^{k_1})) =  \mathcal{NI}(g_1) = \mathcal{NI}(g_2) = f(\M(g^{k_2}))$.
\hfill $\Box$ \\

So, $\{\mathcal{NI}(g):g\in \mathcal{G}\}$ is a partition of 
the set $Y = \{g^y \in \mathcal{NI}(\M(g)): g\in \mathcal{G}\}$. 
Now, we look at the structure of $Y$. 

%We shall consider two cases separately 
%when $p\equiv 3\pmod 4$ and $p\equiv 1\pmod 4$. Let $\R_i$ be the generator set of 
%subgroup of $\mathbb{Z}_p^*$ of order $\frac{p-1}{2^i}$. 

\begin{lemma}\label{lem_y}
 For $g\in \mathcal{G}$,
the number of generators of a subgroup of order 
$\frac{p-1}{2^k}$  in $\mathcal{NI}(g)$ is  
\begin{equation}
  a_k = \left \{
  \begin{aligned}
    & \frac{1}{2^k}|\mathcal{NI}(g)|, && \text{if}\ 1\leq k\leq i-1 \\
    & \frac{1}{2^{k-1}}|\mathcal{NI}(g)|, && \text{if}\ k=i 
    \end{aligned} \right.
\end{equation} 
(Recall $|\mathcal{NI}(g)| = 2^iq_1^{j_1-1}q_2^{j_2-1}$)
\end{lemma}

\noindent Proof: For every $g^y\in \mathcal{NI}(g)$, $(y,q_1q_2) =1$. 
So $a_1 + a_1+ \ldots+ a_i  =  |\mathcal{NI}(g)|$. Let $\frac{\phi(q_1q_2)}{2}a_k = b_k$, 
$1\leq k\leq i$. Since the number of
distinct $\mathcal{NI}$ sets are $\frac{\phi(q_1q_2)}{2}$,   
\begin{eqnarray}\label{eq_y1}
 \frac{\phi(q_1q_2)}{2}(a_1+a_2+\ldots+a_i) & = &  \phi(p-1) \nonumber \\ 
   b_k & \leq & \phi(\frac{p-1}{2^k}) 
\end{eqnarray}
Further, we have  
\begin{eqnarray}\label{eq_y2}
 \sum_{k=1}^i b_k  =   \sum_{k=1}^i \phi(\frac{p-1}{2^k}) = \phi(p-1) \nonumber \\
 \sum_{k=1}^i (\phi(\frac{p-1}{2^k})-b_k)  =   0 
\end{eqnarray}
 From (\ref{eq_y1}) and (\ref{eq_y2}), it follows that, $b_k = \phi(\frac{p-1}{2^k})$, which implies the desired value for $a_k$. \hfill $\Box$  \\

\subsection{Additive inverses of generators of $\mathcal{Z}_p^*$} 
We know that, for any $g_i\in \M(g)$, $g_i^{-1}$ also belongs to $\M(g)$. 
The property of additive inverses is subjected to the prime factorization of $p-1$. For primes of the form
$4k+1$, the additive inverse of a generator in one node belong to the same node. 
The argument is simple: suppose $g_i = g^{\pm x_1+zq_1q_2} \in \M(g)$, then $-g_i = g^{\pm x_1+zq_1q_2 + \frac{p-1}{2}} = 
g^{\pm x_1+q_1q_2(z+\frac{p-1}{2q_1q_2})}\in \M(g)$. So, for every $g\in G$, 
both $g$ and $-g$ belong to the same node. For primes of form $4k+3$, additive inverses of generators belong to sets $\mathcal{NI}$. It is proved 
in the following lemma. 

\begin{lemma}
 Let $p\equiv 3\pmod 4$. Then, for every $g\in \mathcal{G}$, there exists 
 $g_i \in \mathcal{G}$ such that
  $\mathcal{NI}(g_i) = \{ -g: g\in \M(g)\}$.
\end{lemma}
\noindent Proof: For $g^k\in \M(g)$, $g^{k+\frac{p-1}{2}}$ is a residue and it is of order $\frac{p-1}{2^s}$, 
where $s$ is the greatest integer such that 
$2^s| (k+\frac{p-1}{2})$. By Lemma \ref{lem_y}, $g^k\in \mathcal{NI}(g_i)$ for some $g_i\in \G$. Moreover, 
for any $g_z = g^{k+zq_1q_2}$ of  $\M(g)$, $-g_z = g^{k+q_1q_2(z+\frac{p-1}{2})} \in  \mathcal{NI}(g_i)$. 
Since, $|\mathcal{NI}(g_i)| = |M(g)|$ which completes the proof.  \hfill $\Box$ \\

It follows that $\mathcal{NI}(g_i^{x_1^z}) = \{ (-g): g\in \M(g^{x_1^z})\}$. 
Suppose $g_i\in \M(g_j)$. Then $\mathcal{NI}(\M(g_j^{x_1^z})) = \mathcal{NI}(g_1^{x_1^z})$. 
So additive inverses of generators in unicycle 
consisting of vertices $\M(g),\M(g^{x_1}),\ldots ,\M(g^{x_1^{n-1}})$ 
belong to sets $\mathcal{NI}$ corresponding to vertices of unicycle whose vertex set is 
$$\{\M(g_j),\M(g_j^{x_1}), \ldots,\\
\M(g_j^{x_1^{n-1}})\}.$$ 
By this property, we have for every unicycle of D, 
there is a corresponding unicycle of the digraph by sets $\mathcal{NI}$. 
This property is further studied with the help of relation $S$.
\subsubsection{Relation $S$}
 Let $L = \{1,2,\ldots, n\}$ be a labeling on unicycles of the digraph $D$ of missing generators for a 
 prime $p\in P_3$, where $n$ is the number of unicycles in the digraph. 
Let $S$ be a relation on $L$. If additive inverses of generators in 
$i^{th}$ unicycle belong to sets $\mathcal{NI}$ corresponding to vertices in 
$j^{th}$ unicycle, then the pair $(i,j) \in S$. We see that $S$ is either 
reflexive or symmetric.

 \begin{lemma} 
 $S$ is reflexive if and only if $2^n \equiv 1 
 \pmod {q_1q_2}$, $n$ is the number of vertices in a 
 unicycle of the digraph. 
\end{lemma}
\noindent Proof: Suppose $S$ is reflexive. 
So the additive inverse of generator $g$ belong to the set 
$\mathcal{NI}(g^{x_1^j})$ for some $j$ and $gcd(x_1^j\pm (\frac{p-1}{2}+1),p-1)>1$. 
This implies that, $x_1^{2j} \equiv 1 \pmod {q_1q_2}$ and $2^{2j} \equiv 1 \pmod{q_1q_2}$. 
Since $2j$ is the least positive integer such that $2^{2j} \equiv 1 \pmod{q_1q_2}$, $n=2j$. \\
Suppose that, $n$ is the number of vertices in a 
unicycle of the digraph, 
with $2^n \equiv 1 \pmod{q_1q_2}$. Since $n$ is even, $x_1^n \equiv 1 \pmod{q_1q_2}$
 and $-g$ belongs to the set 
 $\mathcal{NI}(g^{x_1^\frac{n}{2}})$. Therefore, $S$ is reflexive  \hfill $\Box$

\begin{lemma}
 If $S$ is not reflexive, then $S$ is symmetric. 
\end{lemma}
\noindent Proof : Suppose unicycles of the digraph are numbered so 
that the generators $g$ and $g^k$ belong to
$n_1^{th}$ and $n_2^{th}$ 
unicycle respectively. Let $-g$ belong to the set $\mathcal{NI}$
corresponding to a vertex 
in $n_2^{th}$ unicycle. So $-g \in \mathcal{NI}(g^{kx_1^j})$ for some $j$.
 Since $gcd(kx_1^j \pm (\frac{p-1}{2}+1),p-1)>1$, it is true that 
$gcd(k\pm x_1^{\phi(q_1q_2)-j},q_1q_2)>1$. Therefore $-g^k$ belongs to 
the set $\mathcal{NI}(g^{x_1^{\phi(q_1q_2)-j \pmod n}})$ 
corresponding to a vertex in $n_1^{th}$ unicycle. Hence $S$ is symmetric. \hfill $\Box$

\begin{cor}
 $S$ is symmetric if and only if $n$ is odd or $2^n \equiv -1 \pmod{q_1q_2}$ for even $n$.
\end{cor}

The above properties of $S$ will be used in devising an algorithm 
which verifies for a given triplet whether there exists a prime in $P_3$. The algorithm 
is presented in section 9.
 
From the table, for primes $p=127$, $139$, $283$ and $907$, $n$ is odd, so the relation $S$ is symmetric. 
For $p=103=2\times 3\times 17+1$, $(3\times 17)|2^8-1$ and therefore $S$ is 
reflexive. For $p = 131=2\times 5\times 13+1$, $S$ is symmetric, 
since $n=6$ is even and $(5\times 13)|(2^6+1)$. 

\section{Map T}\label{sec_t}
\noindent For a prime $p\in P_3$, the digraph formed by the sets of missing generators 
of $\mathcal{Z}_p^*$ is described by a triplet $(c,n,e)$ of 
numbers. The triplet $(c,n,e)$ is unique for $p$. Knowing this fact, we define the map $T: P_3 \rightarrow \mathbb{N}^3$ as   
\begin{displaymath}
 T(p) = (c,n,e).
\end{displaymath}
The map $T$ is not one-one. Since $T(71) = T(79) = (1,12,2)$. 
Indeed one can find many such instances. The map $T$ is not onto and also, $Im(T)\subset
 \mathbb{N}^3$. For a triplet $(c,n,e)\in Im(T)$. It is true that $e$ is 
an even number with at most 3 primes factors and at least one of $c,n$ is even.
 The converse is not true. A triplet $(c_1,n_1,e_1)$ with these properties may not belong to $Im(T)$. 
For example, $(4,1,14) \not \in Im(T)$. Triplet information for some primes is presented in the table
below. 
\begin{center}
\begin{tabular}{|r|r|r||r|r|r|}
 \hline
 $p$ & $(c,n,e)$ & $\phi(p-1)$ & $p$ & $(c,n,e)$ & $\phi(p-1)$ \\
\hline
 31  & (1,4,2)  & 8  & 401 & (0,0,0) & 160\\
\hline
 43 & (2,3,2) & 12 & 409 & (2,8,8) & 128\\
\hline 
67 & (1,10,2) & 20 & 439 & (8,9,2) & 144\\
\hline
71 & (1,12,2) & 24 & 491 & (1,12,14) & 168\\
\hline
79 & (1,12,2) & 24 & 509 & (0,0,0) & 252\\
\hline
89 & (0,0,0) & 40 & 521 & (4,6,8) & 192\\
\hline
131 & (4,6,2) & 48 & 541 & (1,4,36) & 144\\
\hline
151 & (1,4,10) & 40 & 599 & (1,132,2) & 264\\
\hline
181 & (1,4,12) & 48 & 607 & (1,100,2) & 200\\
\hline
199 & (1,10,6) & 60 & 619 & (2,51,2) & 204\\
\hline
223 & (1,36,2) & 72 & 673 & (2,3,32) & 192\\
\hline
241 & (1,4,16) & 64 & 701 & (1,12,20) & 240\\
\hline
281 & (1,12,8) & 96 & 709 & (1,58,4) & 232\\
\hline
311 & (3,20,2) & 120 & 797 & (0,0,0) & 356\\
\hline
349 & (1,28,4) & 112 & 821 & (8,10,4) & 320\\
\hline
397 & (1,10,12) & 120 & 953 & (2,24,8) & 384\\
\hline

\end{tabular} 
\end{center}

It is of interest to ask: 
\begin{center}
{\it 
For a given prime $p\in P_3$,  what is $T(p)?$}
 
\end{center}
%\begin{enumerate}
% \item For a given $(c,n,e)$, does there exist a prime $p\in P_3$ %such that $T(p)=(c,n,e)$? 
%\item   
%\end{enumerate}

Computing $T(p)$ requires factoring $p-1$. Conversely, 
the prime factors of $p-1$ can be computed if $T(p)$ 
is known. 
\begin{theorem}\label{thm_eq}
 Computing $T(p)$ is equivalent to factoring $p-1$. 
\end{theorem}
\noindent Proof:  Suppose $\mathfrak{A}$ is an 
algorithm that outputs $T(p) = (c_1,n_1,e_1)$ for 
a given prime $p\in P_3$. Then, the odd prime factors $r$,$s$ of $p-1$ can 
be computed from the relations: $r\times s = \frac{p-1}{e_1 }$ 
and $(r-1)(s-1) = 2\times n_1 \times e_1$.

Conversely, suppose the prime factors of $p-1$ are known. 
Let $q_1$ and $q_2$ be the prime factors. Then the value of $e$ can be computed as 
$\frac{p-1}{q_1q_2}$. It is known that the variables $c, n$ follow 
 relations:  $c\times n = \frac{1}{2}\phi(q_1q_2)$ and $2^{2n} =1 \pmod{q_1q_2}$. 
So the computation of $n$ requires at most as many operations as 
the computation of the order of $2\pmod {q_1q_2}$ involves. The complexity \cite{ord} of computing the order of an element $\pmod{N}$ 
is $O(\log^4 N)$ when the factorization of $\phi(N)$ is known. 
Thus, the triplet can be computed in $O(\log^4 p)$. \hfill $\Box$  \\

A RSA number is a large integer $N$ that is used as the public parameter in the RSA cryptographic system \cite{rsa}. The security of RSA system is dependent on the difficulty of factoring $N$. The number $N$ is, being the product of two equally sized primes, an 
worst-case instance of Integer Factoring Problem (IFP). 
The best classical technique \cite{nfs} factors $N$ in time sub-exponential in $\log N$. In what follows, we prove that one can factor a given RSA number in time polynomial in $\log N$ using the algorithm $\mathfrak{A}$, under the following number theoretic assumption.  \\

\noindent{\bf Assumption A.} There exists an absolute constant 
$k$ such that, for any odd $N$, 
 the set $\{ 2^iN^j+1: 1\leq i,j< (\log N)^k\}$ contains a prime.

 \begin{lemma}\label{lem_rsa}
 Factors of RSA number $N$ can be computed in time $O(log^{4k+3} N)$.  
\end{lemma}
\begin{proof}
Suppose the assumption A holds true. Then, for a given RSA number $N$, one can search for a prime 
$p\in P_3$ from the set of numbers of the form $2^iN^j+1$ by running PRIMALITY testing algorithm \cite{rabin}. The complexity of verifying if a number $b$ is $O(\log^3 b)$. So, the complexity of finding a prime $p$ is 
\begin{displaymath}
 \sum_{1\leq i,j< (\log N)^k} (\log(2^iN^j+1))^3
 < (\log N)^3 \sum_{1\leq i,j< (\log N)^k} (i+j)^3,
\end{displaymath}
which is roughly $O((\log N)^{4k+3})$. 
Now, the prime $p$ is input to the algorithm $\mathfrak{A}$, 
whose output is $T(p) = (c_1,n_1,e_1)$. 
By Theorem \ref{thm_eq}, the prime factors of $N$ can be computed. \end{proof}

\subsection*{On Assumption A}  

The result in Lemma \ref{lem_rsa} is proved using Assumption A, which is left  
as open issue to be resolved. Our small empirical data, for some randomly chosen odd $N$, suggests that the set $\{ 2^iN^j+1: 1\leq i,j \leq \lceil 5*log_2N 
\rceil \}$ contains a prime. 

\begin{acknowledgement}
Authors would like to gratefully acknowledge the support of National Board for Higher Mathematics, Govt. of India (Sanction order No. 02011/15/2022 NBHM (R.P)/R\&D II/10504)
\end{acknowledgement}

\end{document}